\documentclass{amsart}
\usepackage[dvips]{color}
\usepackage{amssymb}
\usepackage[active]{srcltx}

\author[P. Daskalopoulos]{Panagiota Daskalopoulos}
\address{Panagiota Daskalopoulos\\ Department of Mathematics, Columbia University\\ 2990 Broadway
New York, NY 10027, USA}
\email{pdaskalo@math.columbia.edu}

\author[T. Kuusi]{Tuomo Kuusi}
\address{Tuomo Kuusi\\Aalto University
Institute of Mathematics
\\ P.O. Box 111000
FI-00076 Aalto,
Finland}
\email{tuomo.kuusi@tkk.fi}

\author[G. Mingione]{Giuseppe Mingione}
\address{Giuseppe Mingione\\Dipartimento di Matematica, Universit\`a di Parma\\
Parco Area delle Scienze 53/a, Campus, 43124 Parma, Italy}
\email{giuseppe.mingione@unipr.it.}

\allowdisplaybreaks

\newtheorem{theorem}{Theorem}[section]

\newtheorem{lemma}{Lemma}[section]

\theoremstyle{definition}

\newtheorem{remark}{Remark}[section]

\numberwithin{equation}{section}

\newcommand{\rif}[1]{(\ref{#1})}

\def\eqn#1$$#2$${\begin{equation}\label#1#2\end{equation}}
\def\charfn_#1{{\raise1.2pt\hbox{$\chi
_{\kern-1pt\lower3pt\hbox{{$\scriptstyle#1$}}}$}}}

\def\qq1{q_*}
\def\q2{q_{**}}
\def\dist{\operatorname{dist}}

\def\ep{\varepsilon}

\def\er{\mathbb R}

\newdimen\vintbar
\def\vint{-\kern-\vintbar\int}

\def\0{\boldsymbol 0}

\newcommand{\MM}{\mathcal{M}}

 \newcommand{\mean}[1]{-\hskip-1.08em\int_{#1}}

\newcommand{\trif}[1] {\textnormal{\rif{#1}}}

\newtoks\by
\newtoks\paper
\newtoks\book
\newtoks\jour
\newtoks\yr
\newtoks\pages
\newtoks\vol
\newtoks\publ
\def\et{ \& }
\def\name[#1, #2]{#1 #2}
\def\ota{{\hbox{\bf ???}}}
\def\cLear{\by=\ota\paper=\ota\book=\ota\jour=\ota\yr=\ota
\pages=\ota\vol=\ota\publ=\ota}
\def\endpaper{\the\by, \textit{\the\paper},
{\the\jour} \textbf{\the\vol} (\the\yr), \the\pages.\cLear}
\def\endbook{\the\by, \textit{\the\book},
\the\publ, \the\yr.\cLear}
\def\endpap{\the\by, \textit{\the\paper}, \the\jour.\cLear}
\def\endproc{\the\by, \textit{\the\paper}, \the\book, \the\publ,
\the\yr, \the\pages.\cLear}

\def\mean#1{\mathchoice%
          {\mathop{\kern 0.2em\vrule width 0.6em height 0.69678ex depth -0.58065ex
                  \kern -0.8em \intop}\nolimits_{\kern -0.4em#1}}%
          {\mathop{\kern 0.1em\vrule width 0.5em height 0.69678ex depth -0.60387ex
                  \kern -0.6em \intop}\nolimits_{#1}}%
          {\mathop{\kern 0.1em\vrule width 0.5em height 0.69678ex
              depth -0.60387ex
                  \kern -0.6em \intop}\nolimits_{#1}}%
          {\mathop{\kern 0.1em\vrule width 0.5em height 0.69678ex depth -0.60387ex
                  \kern -0.6em \intop}\nolimits_{#1}}}

\def\vintslides_#1{\mathchoice%
          {\mathop{\kern 0.1em\vrule width 0.5em height 0.697ex depth -0.581ex
                  \kern -0.6em \intop}\nolimits_{\kern -0.4em#1}}%
          {\mathop{\kern 0.1em\vrule width 0.3em height 0.697ex depth -0.604ex
                  \kern -0.4em \intop}\nolimits_{#1}}%
          {\mathop{\kern 0.1em\vrule width 0.3em height 0.697ex depth -0.604ex
                  \kern -0.4em \intop}\nolimits_{#1}}%
          {\mathop{\kern 0.1em\vrule width 0.3em height 0.697ex depth -0.604ex
                  \kern -0.4em \intop}\nolimits_{#1}}}

\newcommand{\aveint}[2]{\mathchoice%
          {\mathop{\kern 0.2em\vrule width 0.6em height 0.69678ex depth -0.58065ex
                  \kern -0.8em \intop}\nolimits_{\kern -0.45em#1}^{#2}}%
          {\mathop{\kern 0.1em\vrule width 0.5em height 0.69678ex depth -0.60387ex
                  \kern -0.6em \intop}\nolimits_{#1}^{#2}}%
          {\mathop{\kern 0.1em\vrule width 0.5em height 0.69678ex depth -0.60387ex
                  \kern -0.6em \intop}\nolimits_{#1}^{#2}}%
          {\mathop{\kern 0.1em\vrule width 0.5em height 0.69678ex depth -0.60387ex
                  \kern -0.6em \intop}\nolimits_{#1}^{#2}}}

\newcommand{\eps}{\varepsilon}

\newcommand{\ern}{\mathbb{R}^n}

\title{Borderline estimates for fully nonlinear elliptic equations}
\begin{document}

\maketitle

\begin{abstract} We prove new borderline regularity results for solutions to fully nonlinear elliptic equations together with pointwise gradient potential estimates.
\end{abstract}

\section{Introduction} The aim of this 
paper is to give borderline regularity results and potential estimates for viscosity solutions to fully nonlinear elliptic equations of the type
\eqn{fully}
$$
F(x,D^2u)=f
$$
in $\Omega$. Here and in the following $\Omega$ will denote  an open subset of $\ern$, $n\geq 2$. The final outcome is that classical sharp results valid for the Poisson equation $-\triangle u =f$ and {\em allowing to find the best function space conditions} on $f$ in order to guarantee certain borderline regularity properties of $Du$  such as boundedness, continuity, BMO/VMO-regularity etc, are extended verbatim to the fully nonlinear case \rif{fully}. This will follow from the fact that classical pointwise estimates for solutions via linear potentials will here find a suitable analog in the fully nonlinear
situation.

First of all, let us introduce the general setting; throughout the paper, according to \cite{CC}, we shall assume the ellipticity and growth conditions
\begin{equation} \label{ell}
    \mathcal{P}^-(X-Y) \leq F(x,X)-F(x,Y) \leq \mathcal{P}^+(X-Y)\,, 
\end{equation}
where $x \in \Omega$ and $X,Y \in \mathcal{S}(n)$ are symmetric square matrices and $\mathcal{P}^-$ and $ \mathcal{P}^+$ are the standard Pucci's extremal operators defined as
\[
 \mathcal{P}^-(X) = \lambda \sum_{\lambda_j > 0} \lambda_j +  \Lambda \sum_{\lambda_j < 0} \lambda_j \,, \qquad
 \mathcal{P}^+(X) = \Lambda \sum_{\lambda_j > 0} \lambda_j +  \lambda \sum_{\lambda_j < 0} \lambda_j\,,
\]
$(\lambda_j)_{j=1}^n$ being the eigenvalues of $X \in \mathcal{S}(n)$ and $\Lambda\geq \lambda>0$. For basic properties of Pucci's operators and basic notation on fully nonlinear elliptic equations we refer to \cite{CC}.  For simplicity we shall consider the additional assumption $F(x,0)=0$, which is actually not restrictive for the kind of results we are going to prove here.

The analysis of fully nonlinear equations can be carried out only starting from a certain regularity of the source term as widely explained in \cite{C, CC, E}. More precisely, a suitable assumption for regularity is
$$f \in L^{p}(\Omega)\,, \qquad \mbox{where}\  p> n_{\rm{E}} \in (n/2,n)$$ and the exponent $n_{\rm{E}}$ depends only on $n$ and the structure/ellipticity constants $\lambda, \Lambda$; this essentially follows joining the basic analysis of Caffarelli \cite{C} - who considered $f \in L^n$, which is natural to apply the ABP principle - with suitable reverse H\"older inequalities valid for 
the fundamental solutions to certain linear elliptic equations, as shown by Escauriaza in~\cite{E}. Henceforth, in this paper, {\em we shall consider only the case $f \in L^p$ for $p > n_{\rm{E}}$}. 
Moreover, in the following, after defining $f\equiv 0$ outside $\Omega$, we shall assume, without loss of generality,  that $f \in L^{p}(\er^n)$.

Our leading regularity assumption on $x \mapsto F(x, \cdot)$, initially assumed to be measurable, generalizes that of plain continuity; we define
$$
\omega(R):= \sup_{\varrho \leq R}   \sup_{y \in \Omega, Y \in \mathcal{S}(n) \setminus \{0\}} \mean{B_{\varrho}(y)} \frac{|F(x,Y)-(F)_{B_\varrho(y)}(Y)|}{|Y|}\, dx\,,
$$
where the averaged operator $(F)_{B_\varrho(y)}(Y)$ is defined as $$ \quad (F)_{B_\varrho(y)}(Y) := \mean{B_{\varrho}(y)} F(x,Y)\, dx\,.$$ Note that the averaged operator satisfies~\rif{ell} whenever the original one $F$ does. We say that $F(\cdot)$ has $\theta$-BMO coefficients for positive $\theta$ if there exists a positive radius $R_\theta$ such that $\omega(R_\theta)\leq \theta$. 
Notice that this formulation of coefficient regularity is slightly different from the original ones considered in \cite{C, CC, S} but the definition considered here leads to similar considerations (observe that $x \mapsto F(x, \cdot)$ is bounded by \rif{ell}).

Finally, we recall that, due to the low integrability of the datum $f \in L^{p}$ and to the low regularity of the $x$-coefficients of $F(\cdot)$, in this paper we shall always be dealing with $L^p$-viscosity solutions of $F(x,D^2u)=f$ in the sense specified in \cite{CCKS, S}.

The type of results we are discussing here are concerned with borderline gradient regularity which can be considered as a limit case of those proved in \cite{C, CC, T1, T2}. The first main result in this respect is the following:

\begin{theorem}[Lorentz space regularity]\label{mainc} Let $u$ be an $L^p$-viscosity solution to the equation \linebreak $F(x,D^2u)=f$, $p>n_E$.
There is $\theta \in (0,1)$ depending only on $n,\lambda,\Lambda$ such that if $F(\cdot)$ has $\theta$-BMO coefficients and $f \in L(n,1)$, that is if
$$
\int_0^{\infty} |\{x\in \Omega \, : \, |f(x)|> \lambda \}|^{1/n}\, d \lambda <\infty\,,
$$
then $Du$ is continuous.
\end{theorem}
The previous result actually comes along with an explicit estimate on the modulus of continuity of $Du$ that can be seen to locally depend only on the $L^\infty$-norm of  $Du$ and on the quantity appearing in the last display; see Remark \ref{modl} below. The sharpness of the previous result is testified by  the borderline character of the space $L(n,1)$ for the case $-\triangle u=f$ with respect to the Lipschitz regularity. Indeed, in general $Du$ is unbounded if $f \in L^{q}$ with $q<n$ while we notice that $L^{q}\subset L(n,1)\subset L^n$ for every $q>n$, with all the inclusions being strict. The exact definition of general Lorentz spaces appears in Section 2 below.

The estimates in this paper are 
motivated by recent results for equations in divergence form, for which pointwise bounds via potentials are available both for $u$ and $Du$ no matter
of the nonlinearity of the operators considered (see for instance \cite{KM, KM1, KM2, TW1, TW2} and \cite{milan} for a survey of results). In particular, for solutions to $-\triangle u=f$ the estimate
\eqn{pot}
$$
 |D u(x)| \leq c \, {\bf I}_{1}^{f}(x,r)  + c\, \mean{B_r(x)}|Du|\, dy$$
holds for a.e. $x$, where ${\bf I}_{1}^{f}(x,r)$ denotes the truncated Riesz potential of $f$ $$
{\bf I}_{1}^f(x,r):= \int_0^r \mean{B_{\varrho}(x)}|f(y)|\, dy\, d\varrho\,,
$$
 see details for example in~\cite[Section 5]{milan}. For fully nonlinear equations, the  analogous estimates do not seem to be known.  Such  lack of estimates is expected as the natural minimal assumption for $f$ is that it belongs to $L^p$, $p>n_{\rm{E}}$. The path to a natural fully nonlinear counterpart for the classical potential estimate and consequent sharp borderline regularity results is instead to consider the $L^p$ version of the classical Riesz potential:
$$
{\bf I}_{1}^f(x,r):= \int_0^r \mean{B_{\varrho}(x)}|f(y)|\, dy\, d\varrho\leq \int_0^r \left(\mean{B_{\varrho}(x)}|f(y)|^p\, dy\right)^{1/p}\, d\varrho =:
\tilde{\bf  I}_{p}^{f}(x,r)\,.
$$
We are going to show that the last quantity - that we consider as a ``modified Riesz potential" - is indeed a suitable replacement of ${\bf I}_{1}^f$ in order to develop potential estimates analogous to the one in~\rif{pot} for solutions to fully nonlinear equations. Especially, it can be used to derive a sharp continuity criterion for the gradient. The results are the following two theorems.

\begin{theorem}[Gradient potential estimate]\label{main1} Let $u$ be an $L^p$-viscosity solution to the equation $F(x,D^2u)=f$ under the assumption \trif{ell}, $p>n_E$. Then, for any $q>n$, there are constants $c \geq 1$ and $\theta \in (0,1)$, both depending only on $n,\lambda, \Lambda,p,q$, such that if $F(\cdot)$ has $\theta$-BMO coefficients, i.e. there is $R_\theta>0$ such that $\omega(R_\theta) \leq \theta$, then
$$
|Du(x_0)|\leq c \, \tilde{\bf  I}_{p}^{f}(x, r) + c \left(\mean{B_r(x_0)} |Du|^q \, dx \right)^{1/q}
$$
holds whenever $B_{r}(x_0) \subset \Omega$, where $x_0$ is a Lebesgue point of $Du$  and $r \leq R_\theta$. Moreover, if $F(\cdot)$ is independent of $x$, no restriction on $r$ is necessary.
\end{theorem}

\begin{theorem}[Gradient continuity via potentials]\label{main2} Assume that  $u$ is  an $L^p$-viscosity solution to the equation $F(x,D^2u)=f$ under the assumption \trif{ell}, $p>n_E$.
There is $\theta \in (0,1)$ depending only on $n,p,\lambda,\Lambda$ such that if $F(\cdot)$ has $\theta$-BMO coefficients and if $\tilde{\bf  I}_{p}^{f}(x,r) \to 0$ as $r \to 0$ uniformly in $x$, then $Du$ is continuous. Moreover, whenever $\Omega' \Subset \Omega'' \Subset  \Omega$ are open subsets, and $\delta \in (0,1]$, the following holds:
\begin{eqnarray}
\nonumber
&& |Du(x_1) - Du(x_2)|
\\ && \qquad
\leq c \|Du\|_{L^\infty(\Omega'')}  |x_1-x_2|^{\alpha(1-\delta)} + c\sup_{x\in \{x_1, x_2\}} \tilde {\bf I}_p^f(x,4|x_1-x_2|^\delta)
\label{contd}
\end{eqnarray}
for all $x_1,x_2 \in \Omega'$, where $c \equiv c(n,p,\lambda,\Lambda,\Omega',\Omega'', \omega(\cdot))$ and $\alpha = \alpha(n,p,\lambda,\Lambda)$.
\end{theorem}
An alternative form of the above continuity estimate, independent of the open subset considered is available in \rif{combiest} below.

Further results follow as a by-product of our estimate. Regularity results in VMO and BMO spaces, reproducing in the fully nonlinear case those known for the Poisson equation, can now be proved. Indeed, the following holds:

\begin{theorem}\label{mainvb} Let $u$ be an $L^p$-viscosity solution to the equation $F(x,D^2u)=f$ under the assumption~\trif{ell}, $p>n_E$. There exists a constant $\theta \in (0,1)$, depending only on $n,\lambda, \Lambda,p$, such that if $F(\cdot)$ has $\theta$-BMO coefficients, then
\eqn{cafforthodox}
$$
\sup_{B_{r}(x_0)} r^{p-n} \int_{B_{r}(x_0)} |f|^p\, dx < \infty
$$
implies that $ Du \in \textnormal{BMO}$ holds locally in $\Omega$. In particular $f \in L(n, \infty)$ implies the local BMO-regularity of $Du$. Moreover, if
\eqn{vaiazero}
$$
\lim_{r\to0} r^{p-n} \int_{B_{r}(x_0)} |f|^p\, dx =0
$$
holds locally uniformly with respect to $x_0$,  then $ Du \in \textnormal{VMO}$ holds locally in $\Omega$.
\end{theorem}
We notice that a conditions as \rif{cafforthodox}-\rif{vaiazero} appear to be dual borderline ones of those considered by Caffarelli \cite{Cann}, who proved $C^{0,\alpha}$-estimates for $Du$ assuming that
$$
\sup_{B_{r}(x_0)} r^{n(1-\alpha)-n} \int_{B_{r}(x_0)} |f|^n\, dx < \infty\,.
$$
Moreover, borderline BMO results for second derivatives have been established in~\cite{CaffDuke}.
See also \cite[Theorem 1.12]{pisa} for an analogous result to Theorem \ref{mainvb} valid  for degenerate quasilinear equations.

Furthermore, the potential estimate in Theorem \ref{main1} plays the role of the classical linear potential estimates via Riesz potentials for equations as $-\triangle u=f$ and indeed it allows to prove refined versions of the classical $W^{1,q}$-estimates for instance in interpolation spaces, something that seems otherwise unreachable with the available techniques. For instance, the next result deals with sharp estimates in interpolation spaces like Lorentz spaces $L(q, \gamma)$ and Morrey spaces $L^{q,s}$ whenever $q>n_{\rm{E}}$, while no interpolation theory seems to be available for fully nonlinear elliptic equations. 
Basic definitions and properties of these
spaces will be given in Section \ref{spazi} below.
\begin{theorem}\label{mainl} Let  $u$ be an $L^p$-viscosity solution to the equation $F(x,D^2u)=f$ under the assumption~\trif{ell}, $p>n_E$. There exists a constant $\theta \in (0,1)$, depending only on $n,\lambda, \Lambda,p$, such that if $F(\cdot)$ has $\theta$-BMO coefficients, 
then
\eqn{lol}
$$f \in L(q, \gamma) \Longrightarrow  Du \in L(nq/(n-q),\gamma)\qquad \mbox{whenever} \ \gamma>0, \ q \in (p,n) $$ and
\eqn{mol}
$$f \in L^{q,s} \Longrightarrow  Du \in L^{\theta q/(s-q),\theta} \qquad \mbox{for} \  p < q< s \leq n$$
 hold locally in $\Omega$.
\end{theorem}
We note that both \rif{lol} and \rif{mol} are sharp already in the case $-\triangle u =f$, and~\rif{mol} is indeed  the fully nonlinear counterpart of some classical results of Adams \cite{adams}. On the other hand, Theorem \ref{mainl} extends  the estimates of \'Swiech \cite{S} to the case of Lorentz and Morrey spaces.

The rest of the paper is now structured as follows; we will first prove Theorems \ref{main1} and \ref{main2}, while  the proof of Theorem \ref{mainvb} will be obtained by modifying the arguments introduced for Theorem \ref{main2}. Finally, Theorems \ref{mainc} and \ref{mainl} will be obtained as a corollary of Theorem \ref{main2} and \ref{main1}, respectively. To conclude this section we briefly describe the notation adopted here; in the following $c$ will denote a general constant larger than one, and relevant dependence on parameters will be emphasized in parentheses, for instance, $c \equiv c(n,p,\lambda, \Lambda)$. The ball of radius $r$ and center $x$ shall be denoted by $B_r(x)$ and if there is no confusion about the center, simply by $B_r$. In the following, given a set $A \subset \er^n$ with positive measure and a map $g \in L^1(A,\er^n)$, we shall denote by
$(g)_A:=\mean{A} g(y)\, dy$
its integral average over the positive measure set $A$. 
As for the notation concerning the regularity theory for fully nonlinear elliptic equations, we adopt the notation fixed in \cite{CC}.

\section{Preliminaries}
In this section we recall a few standard results about viscosity solutions and then prove a basic decay estimate for solutions with coefficients with small oscillations and small data. Finally, in Section \ref{spazi} we recall some basic definitions concerning a few relevant function spaces.

\subsection{Technical preliminary lemmas}
We start recalling a basic result about $W^{1,q}$ estimates for $L^p$-viscosity solutions; its  statement is a  combination of the results in \cite{CC, Cann} and \cite[Theorem 2.1]{S}.
\begin{theorem}\label{thm:Swiech} 
Let $u$ be an $L^p$-viscosity solution to $F(x,D^2u)  = f$ in $B_1$, where $f \in L^p(B_1)$. There is a constant $\theta \in (0,1)$ depending only on $n,p,\lambda,\Lambda$ such that if $\omega(1) \leq \theta$, then $u\in W_{\rm{loc}}^{1,q}(B_1)$ and
\begin{equation}\label{gradient bound}
    \left(\mean{B_{1/2}} |Du|^{q} \, dx \right)^{1/q} \leq c \sup_{B_1} |u| + c\left(\mean{B_{1}} |f|^{p} \, dx \right)^{1/p}
\end{equation}
holds for all $1\leq q < n p/(n-p)$ if $p< n$, and for all $q \geq 1$ otherwise, where the constant $c$ depends only upon $n,p,q,\lambda,\Lambda$.
\end{theorem}
\begin{remark}[Natural scaling] \label{scalare} Let us recall the basic scaling properties of equation $F(x,D^2u)=f$; indeed, if $u$ is an $L^{p}$-viscosity solution of $F(x,D^2u)=f$ in the ball $B_r(x_0) \subset \Omega$, then $\tilde u$ is an $L^{p}$-viscosity solution of $\tilde F(x,D^2\tilde u)=\tilde f$ in $B_1 \equiv B_1(0)$, where
$$
    \widetilde u(x) = \frac{u(x_0 +r x)}{Ar}\,, \quad \widetilde F (x, X ) = \frac{r}{A}  F(x_0+rx , (A/r) X )\,, \quad \widetilde f(x) = \frac{r }{A}f(x_0+rx)\,,
    $$
    and $\widetilde \omega(R) = \omega(r R)$. Here $X \in \mathcal{S}(n)$.
\end{remark}
\begin{lemma}\label{cah} Let $u$ be an $L^p$-viscosity solution to $F(x,D^2u) = f$, $f \in L^p(B_1)$. For every $M\geq 1$  and $\ep>0$ there is a constant $\theta \in (0,1)$ depending only on $n,p,\lambda,\Lambda,M,\ep$ such that if
\begin{equation}\label{beta small 1}
 \|u\|_{L^\infty(B_1)} \leq M\,, \quad \left(\mean{B_1} |f(x)|^p \, dx \right)^{1/p} \leq \theta \,, \quad \textrm{and} \quad \omega(1) \leq \theta
\end{equation}
hold, then there is a function $h$ solving $(F)_{B_1}(D^2 h) = 0$ and satisfying $\|u-h\|_{\infty} \leq \eps$.
\end{lemma}
For the rest of the paper, we consider the {\em excess functional}
$$
E_q(B) := \left(\mean{B} |Du - (Du)_{B}|^q \, dx \right)^{1/q}\,,
$$
with $B \subset \Omega$ denoting a ball with positive radius and $q\geq 1$,  where in general $u$ is a solution of the equation $F(x,D^2u)=f$, and its identity will be clear from the context. We then have the following crucial {\em decay estimate}:
\begin{lemma} \label{main lemma}
Let $u$ be an $L^p$-viscosity solution to $F(x,D^2u) = f$, $f \in L^p(B_1)$. Let $n< q < n p/(n-p)$ when $p<n$ and $q>n$ otherwise. There are constants $\theta,\sigma \in (0,1)$, both depending only on $n,p,\lambda,\Lambda,q$, such that
\begin{equation*}\label{beta small comp}
E_q(B_1) \leq 1\,,  \quad \left(\mean{B_1} |f(x)|^p \, dx \right)^{1/p} \leq \theta\,, \quad \omega(1)\leq \theta
\quad \Longrightarrow \quad
E_q(B_\sigma)\leq 1/3\,.
\end{equation*}
\end{lemma}
\begin{proof} First notice that since $p>n_{\rm{E}}> n/2$,  it follows that when $p<n$ then we have $np/(n-p)>n$. Therefore the condition $n< q < n p/(n-p)$ is non-empty.
Without loss of generality, we may assume that $(u)_{B_1} = 0$ and $(Du)_{B_1} = 0$ (otherwise consider $\widetilde u = u - (u)_{B_1} - (Du)_{B_1}\cdot x$, which solves the same equation as $u$). From Morrey's inequality it follows that
\[
\|u\|_{L^\infty(B_1)} \leq c \left(\mean{B_1} |Du|^q \, dx \right)^{1/q} = c \, E_q(B_1) \leq c \equiv M
\]
with $c \equiv c(n,q)$.
Thus Lemma~\ref{cah} implies that for any $\eps \in (0,M)$ we find $\theta \in (0,1)$ and $h$ solving $(F)_{B_1}(D^2 h) = 0$ such that $\|u-h\|_{\infty} \leq \eps$. Since
\[
\|h\|_{L^{\infty}(B_1)} \leq \|u\|_{L^{\infty}(B_1)}  + \|u-h\|_{L^{\infty}(B_1)} \leq M + \eps \leq 2M\,,
\]
the $C^{1,\alpha}$-regularity estimate (see \cite[Corollary 5.7]{CC} and recall that we are assuming $F(x,0)=0$ that implies $(F)_{B_1}(0) = 0$) for some positive $\alpha$ depending only on $n,\lambda,\Lambda$ gives
$$
\|h\|_{C^{1,\alpha}(B_{1/2})} \leq c \|h\|_{L^{\infty}(B_1)} \leq cM\,.
$$
It follows that there is an affine function $\ell (= h(0) + Dh(0)\cdot x)$ such that
$
\|h-\ell \|_{L^{\infty}(B_{2\sigma})} \leq cM \sigma^{1+\alpha}
$
for all $\sigma \in (0,1/2)$. Now, $u - \ell$ still solves $F(x,D^2u) = f$ and therefore, after scaling in Theorem \ref{thm:Swiech} as in Remark~\ref{scalare}, we obtain
\[
\left(\mean{B_{\sigma}} |D(u-\ell)|^{q} \, dx \right)^{1/q} \leq \frac{c}{\sigma} \|u-\ell \|_{L^{\infty}(B_{2\sigma})}
+ c \sigma \left(\mean{B_{2\sigma}} |f|^{p} \, dx \right)^{1/p}\,.
\]
Inserting here the elementary inequalities
\[
\|u-\ell \|_{L^{\infty}(B_{2\sigma})} \leq \|u-h \|_{L^{\infty}(B_{2\sigma})} + \|h-\ell \|_{L^{\infty}(B_{2\sigma})} \leq \eps + cM \sigma^{1+\alpha}\,,
\]
\[
\left(\mean{B_{2\sigma}} |f|^{p} \, dx \right)^{1/p} \leq c\sigma^{-n/p} \left(\mean{B_{1}} |f|^{p} \, dx \right)^{1/p} \leq c \sigma^{-n/p} \theta\,,
\]
and
\[
\left(\mean{B_{\sigma}} |Du-(Du)_{B_{\sigma}}|^{q} \, dx \right)^{1/q}  \leq 2 \left(\mean{B_{\sigma}} |D(u-\ell)|^{q} \, dx \right)^{1/q}\,,
\]
we conclude that
\[
\left(\mean{B_{\sigma}} |Du-(Du)_{B_{\sigma}}|^{q} \, dx \right)^{1/q}  \leq c(\eps/\sigma + M \sigma^{\alpha} + \sigma^{1-n/p} \theta)\,.
\]
The result follows by taking first $\sigma$ sufficiently small, then $\eps$ small (both depending on $n,p,\lambda, \Lambda,q$ but not on $\theta$) and finally $\theta$ small enough.
\end{proof}
\subsection{Relevant function spaces}\label{spazi} The Lorentz space $L(q,\gamma)(\Omega)$, with $1\leq q< \infty$, $0<  \gamma \leq \infty$ and $\Omega \subseteq \er^n$ being an open subset, is defined prescribing that a measurable map $g$ belongs to $L(q,\gamma)(\Omega)$ iff
\eqn{lonorm}
$$\int_0^{\infty}
\left(\lambda^q|\{x \in \Omega \ : \ |g(x)|> \lambda \}|\right)^{\gamma/q} \, \frac{d\lambda}{\lambda}< \infty$$
when $\gamma < \infty$; for $\gamma=\infty$ the membership to $L(q,\infty)(\Omega)\equiv \MM^{q}(\Omega)$ is 
instead settled by
\eqn{dema}
$$
\sup_{\lambda > 0} \, \lambda^q|\{x \in \Omega \ : \ |g(x)|> \lambda\}|< \infty\;.
$$
This is the so-called Marcinkiewicz, or weak-$L^q$ space. It readily follows for the definitions above that
\eqn{basic}
$$
f \in L(q, \gamma) \Longrightarrow |f|^p \in L(q/p, \gamma/p) \qquad \mbox{for}\ p \leq q\,.
$$
As for Morrey spaces we have that a map $g \colon \er^n \to \er$ 
 belongs to the Morrey space $L^{q, s}$ for $q \geq 1$ and $0 \leq s \leq n$ iff
$$
\sup_{B_{\varrho} \subset \er^n}\, \varrho^{s}\mean{B_\varrho} |g|^q\, dx< \infty\,.
$$
We finally recall the definition of BMO (Bounded Mean Oscillation) and VMO (Vanishing Mean Oscillation) spaces. For a possibly vector valued map $g\in L^1(\Omega)$, define
$$
\omega_g(R):= \sup_{B_\varrho \subset \Omega, \varrho \leq R} \mean{B_{\varrho}} |g-(g)_{B_{\varrho}}|\, dx\,.
$$
If $\limsup_{R\to 0} \omega_g(R) < \infty$ then $g \in$ BMO$(\Omega)$; if $\lim_{R\to 0} \omega_g(R) =0$ then $g \in$ VMO$(\Omega)$. BMO and VMO functions have been introduced in \cite{JN} and \cite{Sa}, respectively. The borderline role of BMO stems not only by the (strict) inclusions $L^{\infty}\subset$ BMO $\subset L^{q}$ which 
hold true for every $q< \infty$, but also by the relevant role this space plays in interpolation theory and in problems with critical growth, where in many situations it properly replaces $L^\infty$. 

We remark that the local versions of all the spaces above can be defined in an 
obvious manner by saying that whenever $X$ denotes a function space of the ones considered above, then $g \in X(\Omega)$ locally iff $g \in X(\Omega')$ whenever $\Omega' \Subset \Omega$.

\section{Proof of results}

\begin{proof}[Proof of Theorem \ref{main1}]
Fix $q>n$. Accordingly, we find $\theta,\sigma \in (0,1/2)$ as in Lemma~\ref{main lemma} depending only on $n,p,\lambda,\Lambda, q$. The assumptions of  the Theorem guarantee that for this $\theta$ there is $R_\theta>0$ such that $\omega(R_\theta) \leq \theta$. Next, we consider the sequence of shrinking balls $B_{i}:= B_{r_i}$ whenever $i \geq 0$ is an integer and $r_i := \sigma^ir/2$. The radius $r$ satisfies $r<R_\theta$.
Define now, for a positive parameter $\tilde \eps$,
$$
A_{i}:= E_{q}(B_i)  + \frac{r_i}{\theta} \left(\mean{B_i} |f(x)|^p \, dx \right)^{1/p} +
\tilde \eps\,.
$$
Consider the scaled solution in $\tilde u_i\equiv \tilde u$ defined according to the scaling described in Remark~\ref{scalare}, with $A \equiv A_i$ and $r\equiv r_i$. We can apply Lemma~\ref{main lemma}, using also assumption $\omega(r)\leq \theta$, and, after scaling back to $u$ in $B_i$ and eventually letting $\tilde \eps \to 0$, get
\eqn{servedopo}
$$
E_{q}(B_{i+1})\leq \frac13 E_{q}(B_{i})  + \frac{1}{3\theta} r_i \left(\mean{B_i} |f(x)|^p \, dx \right)^{1/p}
$$
for every $i \geq 0$. Adding up the previous inequalities yields
$$
\sum_{i=1}^{j+1}E_{q}(B_{i}) \leq \frac13 \sum_{i=0}^{j}E_{q}(B_{i}) +  \frac{1}{3\theta}\sum_{i=0}^{j} r_i \left(\mean{B_i} |f(x)|^p \, dx \right)^{1/p}
$$
whenever $j \geq 0$ so that
\eqn{servedopo2}
$$
\sum_{i=0}^{j+1}E_{q}(B_{i}) \leq \frac32 E_q(B_0) + \frac{1}{2\theta}\sum_{i=0}^{\infty} r_i \left(\mean{B_i} |f(x)|^p \, dx \right)^{1/p}$$
follows.
Using H\"older's inequality we also get
\begin{eqnarray}
\nonumber  \sum_{i=0}^{j+1}|(Du)_{B_{i+1}}-(Du)_{B_{i}}| &  \leq &  \sum_{i=0}^{j+1}  \left(\mean{B_{i+1}} |Du - (Du)_{ B_i}|^q \, dx \right)^{1/q}\\
 &  \leq &  \sigma^{-n/q}\sum_{i=0}^{j+1} E_q(B_i) \,. \label{servedopo3}
\end{eqnarray}
Now \rif{servedopo2}-\rif{servedopo3} yield, whenever $j>0$,
\begin{eqnarray*}
 |(Du)_{B_{j+1}}|  &\leq &   \sum_{i=0}^{j}|(Du)_{B_{i+1}}-(Du)_{B_{i}}|  + |(Du)_{B_0}| 
\\ &\leq  &  \sigma^{-n/q}\sum_{i=0}^{j} E_q(B_i) + |(Du)_{B_0}|  \\
&  \leq &  c \sum_{i=0}^{\infty} r_i \left(\mean{B_i} |f(x)|^p \, dx \right)^{1/p}+ c \left(\mean{B_r} |Du|^q\, dx \right)^{1/q}
\end{eqnarray*} for a constant $c$ depending now only on $n,p,\lambda, \Lambda, q$. Since $x_0$ is a Lebesgue point for $Du$, we obtain
$$
|Du(x_0)| =   \lim_{j \to \infty} |(Du)_{B_{j+1}}|  \leq    c\sum_{i=0}^{\infty} r_i \left(\mean{B_i} |f(x)|^p \, dx \right)^{1/p}+c \left(\mean{B_r} |Du|^q\, dx \right)^{1/q}.
$$
The assertion hence follows since
\begin{eqnarray}  \nonumber
&& \sum_{i=0}^{\infty} r_i \left(\mean{B_i} |f(x)|^p \, dx \right)^{1/p}  = (\log 2)^{-1} r_ 0 \left(\mean{B_{0}} |f(x)|^p \, dx \right)^{1/p} \int_{r/2}^{r} \frac{d\varrho}{\varrho}
\\ \nonumber && \qquad \qquad \quad +
  (-\log \sigma)^{-1} \sum_{i=0}^{\infty}r_ i \left(\mean{B_{i+1}} |f(x)|^p \, dx \right)^{1/p}\int_{\sigma^{i+1}r/2}^{\sigma^i r/2} \frac{d\varrho}{\varrho}
\\  && \qquad \quad  \ \ \leq \left[(\log 2)^{-1} 2^{-n/p} + (-\log \sigma)^{-1} \sigma^{-n/p} \right] \int_{0}^r \left(\mean{B_\varrho} |f(x)|^p \, dx \right)^{1/p} \, d\varrho =
 c \tilde{\bf I}_{p}^{f}(x, r)\,, \label{tri}
\end{eqnarray}
where $c$ depends only on $n,p,\lambda,\Lambda,q$ as also $\sigma$ depends only on these parameters.
\end{proof}
\begin{proof}[Proof of Theorem \ref{main2}]
Let $\Omega' \Subset \Omega'' \Subset \Omega$ be open subsets and $R_d := \dist(\Omega',\Omega'')$ and recall that by Theorem \ref{main1} and a standard covering argument it follows that $Du$ is locally bounded in $\Omega$; therefore $\|Du\|_{L^{\infty}(\Omega'')}< \infty$ and this number will stay fixed throughout all the rest of the proof.

The proof goes in two steps: we first prove that $Du$ is locally VMO-regular; then, using this fact and the convergence of $\tilde {\bf I}_p^{f}(x,r)$ to zero as $r\to 0$ uniformly in $x$, we shall prove that $Du$ is continuous. In the following we keep the terminology and the notation introduced for the proof of Theorem \ref{main1} above, in which we now fix $p$ and $q>n$, $\theta,\sigma \in (0,1/2)$, both depending only on $n,p,\lambda,\Lambda, q$; moreover, we let $R:= \min\{R_d, R_\theta, 1\}/2$.

{\em Step 1: Local VMO-regularity of $Du$}. We take $r \in (0,R)$ and fix for the moment $\tau \in [\sigma r,r]$; note that if $x_0 \in \Omega'$ then $B_{2r}\equiv B_{2r} (x_0) \subset \Omega''$. Define next  the sequence of shrinking  balls $B_{i}:= B_{r_i}(x_0)$ with $r_i := \sigma^{i} \tau$.
Iterating~\eqref{servedopo} gives 
\eqn{interm0}
$$
 E_{q}(B_{i+1})  \leq  \frac{1}{3^{i+1}}E_{q}(B_{0})  + \frac{1}{3\theta} \sum_{k=0}^{i}  \frac{r_{i-k}}{3^k} \left(\mean{B_{i-k}} |f(x)|^p \, dx \right)^{1/p}
$$
for all integers $i\geq 0$ and with $c\equiv c(n,p,\lambda,\Lambda,q)$, so that, proceeding as in~\eqref{tri} to estimate the last sum, we obtain
\eqn{interm}
$$
 E_{q}(B_{i+1}) \leq  \frac{2}{3^{i+1}} \|Du\|_{L^{\infty}(\Omega'')} + \frac{c}{\theta}\,  \tilde {\bf I}_p^f(x_0,2r)\,.
 $$
We are now going to prove that, for a suitable constant $\tilde c \equiv \tilde c(n,p,q,\lambda, \Lambda)$ the following inequality holds whenever $0 < \varrho \leq r \leq R$:
\begin{equation} \label{eq:VMO}
E_{q}(B_\varrho(x_0)) \leq \tilde c \left[ \left( \frac{\varrho}{r} \right)^{\alpha} \|Du\|_{L^{\infty}(\Omega'')}  + \tilde {\bf I}_p^f(x_0,2r) \right]
\end{equation}
where $\alpha := -\log 3/\log\sigma$. Let us first show how the inequality in the previous display implies the VMO-regularity of the gradient under the assumptions of the theorem; indeed, choose $\eps> 0$ and determine $R_1< R$, depending also on $\eps$, such that $\tilde c\tilde {\bf I}_p^f(x_0,2R_1) \leq \eps/2$, uniformly in $x_0$; then choose $R_2< R_1$ again depending on $\eps$, such that $\tilde c(R_2/R_1)^\alpha \|Du\|_{L^{\infty}(\Omega'')}\leq \eps/2$; it turns out that $E_{q}(B_\varrho(x_0)) \leq \eps$ whenever $B_\varrho(x_0) \subset \Omega'$ and $\varrho \leq R_2$ and this in fact means that $Du$ is locally VMO-regular. Notice that $R_2$ depends only on $n,p,q,\lambda, \Lambda,\eps$ and $\|Du\|_{L^{\infty}(\Omega'')}$. It thus remains to show the validity of \rif{eq:VMO}. To this aim notice that it is sufficient to show \rif{eq:VMO} for $\varrho < \sigma r $; indeed, in the case $\varrho  \in  [\sigma r, r]$ estimate \rif{eq:VMO} trivially follows by estimating
$$E_{q}(B_\varrho(x_0)) \leq \sigma^{-n/q-\alpha} (\varrho/r)^{\alpha} E_{q}(B_r(x_0))=: \tilde c/2 (\varrho/r)^{\alpha} E_{q}(B_r(x_0)) \leq \tilde c (\varrho/r)^{\alpha}  \|Du\|_{L^{\infty}(\Omega'')}.$$
 We finally analyze the case $\varrho < \sigma r $, when there obviously exists $i \geq  0$ such that $\sigma^{i+2}r \leq \varrho \leq \sigma^{i+1}r$ and so that we can write $\varrho = \sigma^{i+1} \tau$ for some $\tau \in [\sigma r, r]$. At this stage \rif{eq:VMO} follows directly by \rif{interm} with the corresponding choice of $\tau$. Note that the crucial point here is that \rif{interm} actually represents a family of inequalities for the families $B_{i}:= B_{\sigma^i\tau}(x_0)$, and such inequalities hold uniformly with respect to the choice of $\tau \in [\sigma r, r]$ (and $x_0 \in \Omega'$).

{\em Step 2: Continuity of $Du$}. Let $0<\varrho <\tau < r< R$ and let $j\geq 0$ be an integer such that $\sigma^{j+1}\tau \leq \varrho < \sigma^j \tau$. Define the dyadic sequence of balls as $B_i := B_{r_i}(x_0)$, $i=0,1,\ldots$, where $r_i = \sigma^i \tau$.
Estimating
\[
|(Du)_{B_\varrho(x_0)} - (Du)_{B_j}|  \leq \mean{B_\varrho(x_0)} |Du-(Du)_{B_j}| \, dx \leq  \sigma^{-n/q} E_q(B_j)
\]
leads to (we consider the case $j \geq 1$ otherwise we use the previous estimate)
\begin{eqnarray*}
&& |(Du)_{B_\varrho(x_0)} - (Du)_{B_\tau(x_0)}| \\ && \quad  \leq  |(Du)_{B_\varrho(x_0)} - (Du)_{B_j}|
+ \sum_{i=0}^{j-1}|(Du)_{B_{i+1}}-(Du)_{B_{i}}| \leq 2 \sigma^{-n/q} \sum_{i=0}^{\infty} E_q(B_j)\,,
\end{eqnarray*}
where we applied \rif{servedopo3}.
Recalling~\rif{servedopo2} - letting $j \to \infty $ there - and~\rif{tri} we get
\[
|(Du)_{B_\varrho(x_0)} - (Du)_{B_\tau(x_0)}| \leq  c E_q(B_\tau(x_0)) + c \, \tilde {\bf I}_p^f(x_0,2\tau)\,.
\]
In turn, merging this with~\eqref{eq:VMO} gives
\eqn{letr}
$$
|(Du)_{B_\varrho(x_0)} - (Du)_{B_\tau(x_0)}| \leq  c \left[\left( \frac{\tau}{r} \right)^{\alpha} \|Du\|_{L^{\infty}(\Omega'')}  + \tilde {\bf I}_p^f(x_0,2r) \right]
$$
for all $x_0 \in \Omega'$ and $0<\varrho<\tau<r<R$.
Fix now $\eps>0$. Using assumptions of the theorem we first find $R_3$ such that $c \tilde {\bf I}_p^f(x_0,2R_3) \leq \eps/2$ and then $R_4< R_3$ such that $c (R_4/R_3)^\alpha \|Du\|_{L^{\infty}(\Omega'')} \leq \eps/2$ holds. We conclude that
$$|(Du)_{B_\varrho(x_0)} - (Du)_{B_\tau(x_0)}| \leq \eps$$ provided $\varrho, \tau \leq R_4$, 
for every $x_0 \in \Omega'$. This implies that $((Du)_{B_s(x_0)})_{s< R}$ is a Cauchy net - uniformly with respect to $x_0 \in \Omega'$ - and as such the limit $ \lim_{s \to 0} (Du)_{B_s(x_0)}$ exists for {\em every} $x_0 \in \Omega'$. Notice that in this way $Du$ becomes pointwise defined at every point since, as usual, it can be identified with its precise representative. Moreover,  since the above argument is uniform in $x_0 \in \Omega'$ and the maps $x_0 \mapsto (Du)_{B_s(x_0)}$ are continuous for each fixed $s$, we immediately conclude that $Du$ is continuous in $\Omega'$ being the uniform limit of continuous maps. Finally, since the choice of $\Omega'$ is arbitrary we conclude that $Du$ is continuous in $\Omega$. It remains to prove \rif{contd}, that is, it remains to give an estimate for the modulus of continuity of $Du$. To this aim we start letting $\varrho \to 0$ in \rif{letr}, thereby getting
\begin{equation}\label{eq:pointwise}
|Du(x_0) - (Du)_{B_\tau(x_0)}| \leq  c \left[\left( \frac{\tau}{r} \right)^{\alpha} \|Du\|_{L^{\infty}(\Omega'')}  + \tilde {\bf I}_p^f(x_0,2r) \right]
\end{equation}
for all $x_0\in \Omega'$ and $0<\tau<r<R$. Next, we consider $x_1,x_2 \in \Omega'$ such that $|x_1-x_2| < R/4$ and let
$\tau = 2|x_1-x_2|$. We first estimate
\begin{eqnarray}
\nonumber
|Du(x_1) - Du(x_2)| & \leq & |Du(x_1) - (Du)_{B_{\tau/2}(x_1)}| + |Du(x_2) - (Du)_{B_{\tau/2}(x_2)}|
\\ \nonumber && \qquad \qquad + |(Du)_{B_{\tau/2}(x_1)} - (Du)_{B_{\tau/2}(x_2)}| =: \mathcal{T}_1 + \mathcal{T}_2 + \mathcal{T}_3\,.
\end{eqnarray}
The first two terms are bounded above by~\eqref{eq:pointwise} as follows
\[
\mathcal{T}_1 + \mathcal{T}_2 \leq c \left[\left( \frac{\tau}{r} \right)^{\alpha} \|Du\|_{L^{\infty}(\Omega'')}  + \tilde {\bf I}_p^f(x_1,2r) + \tilde {\bf I}_p^f(x_2,2r) \right]
\]
where $r > \tau$ 
satisfies $r < R$.
For $\mathcal{T}_3$ we instead use \eqref{eq:VMO} and get, again for $\tau<r$
\begin{eqnarray}
\nonumber \mathcal{T}_3 & \leq &  \left(\mean{B_{\tau/2}(x_1)} |Du-(Du)_{B_{\tau/2}(x_2)}|^q \, dx \right)^{1/q}
\\ \nonumber & \leq & c E_q(B_{\tau}(x_2)) \leq 2 \left[\left( \frac{\tau}{r} \right)^{\alpha} \|Du\|_{L^{\infty}(\Omega'')}  + \tilde {\bf I}_p^f(x_2,2r) \right]\,.
\end{eqnarray}
Combining the content of the 
last three displays gives that
\begin{eqnarray}
&& 
|Du(x_1) - Du(x_2)|
\leq  c \left[\left( \frac{|x_1-x_2|}{r} \right)^{\alpha} \|Du\|_{L^{\infty}(\Omega'')}  + \tilde {\bf I}_p^f(x_1,2r) + \tilde {\bf I}_p^f(x_2,2r)  \right]
\label{combiest}
\end{eqnarray}
holds whenever $x_1,x_2 \in \Omega'$ are such that $|x_1-x_2| \leq  R/4$; we recall that $R= \min\{R_d, R_\theta, 1\}/2$ is fixed in the beginning of the proof. To arrive at \rif{contd} we start by taking $r=(2|x_1-x_2|/R)^\delta R/2$, which gives by~\eqref{combiest} that
\begin{eqnarray*}
& & |Du(x_1) - Du(x_2)| \leq  c \left[ \left( \frac{|x_1-x_2|}{R} \right)^{(1-\delta)\alpha} \|Du\|_{L^{\infty}(\Omega'')}  + \sup_{x\in \{x_1,x_2\}} \tilde {\bf I}_p^f(x,2|x_1-x_2|^\delta )   \right]\,,
\end{eqnarray*}
and consequently \rif{contd} follows with a new constant $c \leftrightarrow2c R^{-\alpha}$ under the condition $|x_1-x_2| \leq R/4$. Notice that here is the point where the additional dependence of the constant $c$ in \rif{contd} on $\omega(\cdot),\Omega',\Omega''$ appears via the definition of~$R$. In the case $|x_1-x_2| \geq R/4$ we instead obtain~\rif{contd} by simply estimating
\begin{eqnarray*}
|Du(x_1) - Du(x_2)| & \leq & (R/4)^{-\alpha(1-\delta)} |x_1-x_2|^{\alpha(1-\delta)} (|Du(x_1)|+|Du(x_2)|) \\ &\leq & 8 \, R^{-\alpha} |x_1-x_2|^{\alpha(1-\delta)}\|Du\|_{L^{\infty}(\Omega'')}\,.
\end{eqnarray*}
This finishes the proof.
\end{proof}
\begin{proof}[Proof of Theorem \ref{mainvb}] The proof is a straightforward consequence of the arguments used for Theorem \ref{main2}, Step 1, to which we refer for complete notation used here. Indeed, notice that after \rif{interm0}, instead of getting \rif{interm}, we can also estimate as
$$
E_{q}(B_{i+1}) \leq  \frac{2}{3^{i+1} r^{n/q}}\|Du\|_{L^{q}(\Omega'')} + \frac1{\theta}\sup_{\varrho < r} \,  \left(\varrho^{p-n}\int_{B_\varrho} |f(x)|^p \, dx \right)^{1/p}\,.
$$
and, then, as for \rif{eq:VMO}, we can prove that
\begin{equation} \label{eq:VMO2}
E_{q}(B_\varrho(x_0)) \leq \tilde c \left[ \frac{1}{r^{n/q}} \left( \frac{\varrho}{r} \right)^{\alpha} \|Du\|_{L^{q}(\Omega'')} + \sup_{\varrho < r} \,  \left(\varrho^{p-n}\int_{B_\varrho} |f(x)|^p \, dx \right)^{1/p} \right]
\end{equation}
holds whenever $0 < \varrho \leq r \leq R$, again for a constant $\tilde c$ depending only on $n,p,q,\lambda, \Lambda$.
At this stage in order to prove the VMO-regularity of $Du$ we proceed as for Theorem \ref{main2}, Step 1, but using \rif{vaiazero} instead of the fact that $\tilde {\bf I}_p^f(x_0,r)\to 0$ uniformly with respect to $x_0$ when $r\to 0$. As for the BMO-regularity, this immediately follows by \rif{eq:VMO2} once \rif{cafforthodox} is assumed. Finally, observe that in turn \rif{cafforthodox} is implied by the condition $f \in L(n,\infty)$ thanks to the classical H\"older type inequality for Marcinkiewicz spaces $L(n, \infty)$ (compare with the definition in \rif{dema}):
$$
\int_{B_{r}(x_0)} |f|^p\, dx\leq \frac{ \omega_n^{1-p/n}n}{n-p} r^{n-p} \left(\sup_{\lambda\geq 0} \, \lambda|\{x \in B_r(x_0)\, : \, |f|> \lambda\}|^{1/n}\right)^p
$$
which is valid whenever $1\leq p<n$ (see for instance \cite{pisa}).
\end{proof}

\begin{proof}[Proof of Theorem \ref{mainc}] The result is actually a corollary of Theorem \ref{main2} used for a choice $p<n$, once a few basic facts about Lorentz spaces are used. For this we recall that a basic maximal-type characterization of such spaces claims that $g \in L(q,\gamma)$ for $q>1$ and $\gamma>0$ iff
\eqn{mc}
$$
\int_0^\infty \left(g^{**}(\varrho)\varrho^{1/q}\right)^\gamma\, \frac{d\varrho}{\varrho}< \infty\,,$$
where $g^{**}(\varrho):= \varrho^{-1}\int_0^\varrho g^*(t)\, dt$ and $g^*\colon [0,\infty) \to [0,\infty)$ is the non-increasing rearrangement of $g$, that is $g^*(s):=\sup\, \{t\geq 0 \, : \, |\{x \in \er^n \, : \, |g(x)|>t\}|>s\}$ (see for instance~\cite{G}). Now let $g :=|f|^p$ and observe that, for every ball
$B_{\varrho}$, we have by the classical Hardy-Littlewood inequality~\cite{HLP} that
$$
\mean{B_{\varrho}(x_0)}g(y)\, dy \leq \frac{1}{\omega_n\varrho^n}\int_0^{\omega_n\varrho^n} g^{*}(t)\, dt \leq g^{**}(\omega_n\varrho^n)\,,
$$
where $\omega_n$ denotes the measure of the unit ball in $	\er^n$. Integrating the previous inequality yields in turn
$$
\tilde{\bf  I}_{p}^{f}(x,r)\leq   \int_0^r [g^{**}(\omega_n\varrho^n)]^{1/p} \, d\varrho
$$
and changing variables we conclude with
\eqn{convi}
$$
\sup_x \, \tilde{\bf  I}_{p}^{f}(x,r)\leq   \frac{1}{\omega_n^{1/n}}\int_0^{\omega_nr^n} [g^{**}(\varrho)\varrho^{p/n}]^{1/p} \, d\varrho\,.
$$
Now \rif{basic} in particular implies that if $f \in L(n,1)$ then,  $|f|^p = g\in L(n/p,1/p)$ so that the quantity in the right hand side of~\rif{convi} is finite. Keep in mind that this is the point where we need to take $p<n$ in order to use the characterization in \rif{mc} with $q =n/p>1$. As such,   it tends to zero as $r\to 0$ and therefore we infer   that $\tilde{\bf  I}_{p}^{f}(x,r)\to 0$ uniformly with respect to $x$ when $r\to 0$. At this stage we can apply Theorem~\ref{main2} to conclude that $Du$ is continuous and the proof is complete.
\end{proof}

\begin{remark}\label{modl} Combining \rif{contd} and \rif{convi} yields a modulus of continuity for $Du$ in terms of the Lorenz norm of $|f|^p$. Indeed, in the setting of Theorem \ref{main2} we get
\begin{eqnarray*}
&&|Du(x_1) - Du(x_2)| \\ &&
\quad \leq c \|Du\|_{L^\infty(\Omega'')} |x_1-x_2|^{\alpha(1-\delta)}  + c \int_0^{\omega_n4^n|x_1-x_2|^{n\delta}} [(|f|^p)^{**}(\varrho)\varrho^{p/n}]^{1/p} \, d\varrho
\end{eqnarray*}
for all $x_1,x_2 \in \Omega'$, $p<n$, $\delta \in [0,1]$, where $c \equiv c(n,p,\lambda,\Lambda,\Omega',\Omega'', \omega(\cdot))$ and $\alpha = \alpha(n,p,\lambda,\Lambda)$.
\end{remark}

\begin{proof}[Proof of Theorem \ref{mainl}] The idea is to use Wolff and Havin-Maz'ya potentials to reduce the study of properties of the modified Riesz potential $\tilde{\bf  I}_{p}^{f}$ to the one of the usual Riesz potentials and then apply standard results on their mapping properties to deduce suitable mapping properties of $\tilde{\bf  I}_{p}^{f}$; then estimates \rif{lol}-\rif{mol} follow by Theorem \ref{main1} and a standard covering argument. Indeed we notice that
$$
\tilde{\bf  I}_{p}^{f}(x,r)= \omega_n^{-1/p}{\bf W}^{|f|^{p}}_{p/(p+1),p+1}(x,r)\,, \qquad \
$$
where the Wolff potential ${\bf W}^{\mu}_{\beta, p+1}$ of a measure $\mu$ - and therefore of an $L^1$-function - $\mu$ is defined by
$$
{\bf W}^{\mu}_{\beta, p+1}(x,r):= \int_0^r \left(\frac{|\mu|(B(x,\varrho))}{\varrho^{n-\beta (p+1)}}\right)^{1/p}\, \frac{d\varrho}{\varrho}\qquad \qquad  \beta \in (0,n/(p+1)]\,.
$$
In turn, a classical fact established by Havin \& Maz'ya \cite{MH} is that, for the range of exponents $q(=p+1)$ which is of interest here, Wolff potentials can be controlled by so called Havin-Mazya potentials ${\bf V}_{\beta, p+1}(\mu)$ in the sense that the following inequality holds whenever $\beta(p+1)<n$:
\eqn{keyest}
$$
{\bf W}_{\beta, p+1}^{\mu}(x,r) \leq c(n,p,\beta){\bf V}_{\beta, p+1}(\mu)(x):=  c(n,p,\beta)I_{\beta}\left[\left(I_{\beta}(|\mu|)\right)^{1/p}\right](x)\,,
$$
where on the right hand side there appears the standard Riesz potential on $\er^n$:
\eqn{ridef}
$$
I_{\beta}(\mu)(x):= \int_{\er^n} \frac{d \mu(y)}{|x-y|^{n-\beta}}\qquad \qquad  \qquad  \beta \in (0,n]\,.
$$
Therefore we have
\eqn{IV}
$$
\tilde{\bf  I}_{p}^{f}(x, r)\lesssim {\bf V}_{p/(p+1), p+1}(|f|^p)(x) \qquad \forall \ r>0 \,.
$$
Using the previous inequality we deduce that
\eqn{imb}
$$
\tilde{\bf  I}_{p}^{f}(\cdot, r) \in L(nq/(n-q),\gamma)\qquad \forall \ r >0
$$
holds locally by the well-known mapping property of Riesz potentials in Lorentz spaces (see for instance \cite{minann}), that is $$I_{\beta} \colon  L(t, \gamma) \to L(nt/(n-\beta t), \gamma)\qquad \mbox{for}\ \beta t < n, \ t>  1,\ \gamma >0\,,$$ used first with $t =q/p>1$ and then with $t= nq(p+1)/[n(p+1)-q])$, with $\beta = p/(p+1)$, and yet using \rif{basic} repeatedly. At this stage \rif{lol} follows from \rif{imb}, Theorem \ref{main1}
- again applied with some $p<n$ in the range $p \in (n_{\rm{E}}, n)$ - and a standard covering argument. It remains to prove \rif{mol}, for which we follow a similar path, but this time relying on the following basic result due to Adams \cite{adams} describing the behavior of Riesz potentials with respect to Morrey spaces:
\eqn{prima}
$$
g \in L^{t, s} \Longrightarrow I_{\beta}(|g|)\in  L^{s t/(s-\beta t ), s} \qquad \qquad 1<t < s/\beta\,.
$$
The result in \rif{mol} now follows by \rif{IV} and \rif{prima} as for the case of \rif{lol} and using, repeatedly, the fact that $
g \in L^{q,s}$ implies $|g|^p \in L^{q/p,s}$ for $p \leq q$.
\end{proof}

\subsection*{Acknowledgements.} The authors are supported by the NSF grant  0604657, by the ERC grant 207573 ``Vectorial Problems" and by the
Academy of Finland project ``Potential estimates and applications for nonlinear parabolic partial
differential equations".

\end{document}